\documentclass[conference, 10pt, twocolumn]{ieeeconf}       
\bstctlcite{bstctl:etal, bstctl:nodash, bstctl:simpurl}
\def\BibTeX{{\rm B\kern-.05em{\sc i\kern-.025em b}\kern-.08em
    T\kern-.1667em\lower.7ex\hbox{E}\kern-.125emX}}
    
\usepackage[left=54pt, right=54pt,bottom=54pt, top=54pt]{geometry}

\usepackage{amsmath,mathrsfs,amsfonts,amssymb,graphicx,epsfig}
 \usepackage{amsthm}
\usepackage{subcaption}
\usepackage{color,multirow,rotating,tcolorbox}    
\usepackage{algorithmic,algorithm}
\usepackage{cite,url,framed,bm,balance,nicematrix,physics}
\usepackage{stmaryrd,mathtools}
\usepackage[dvipsnames]{xcolor}
\let\labelindent\relax
\usepackage{enumitem}
\usepackage{balance}

\newtheorem{theorem}{Theorem}

\newtheorem{proposition}{Proposition}

\newtheorem{mytheorem}{Theorem}[]      

\usepackage{tikz}
\usetikzlibrary{calc,trees,positioning,arrows,chains,shapes.geometric,decorations.pathreplacing,decorations.pathmorphing,shapes, matrix,shapes.symbols}

\usepackage[breaklinks]{hyperref}

\hypersetup{
    colorlinks=true,
    linkcolor=blue,
     citecolor=blue,
    pdftitle={Overleaf Example},
    pdfpagemode=FullScreen,
    }

\definecolor{DarkGreen}{RGB}{46,139,87}

\newcommand{{\dx}}{{{\rm d} x}}
\newcommand{{\dy}}{{{\rm d} y}}
\newcommand{{\dt}}{{{\rm d} t}}

\makeatletter
\newcommand{\pushright}[1]{\ifmeasuring@#1\else\omit\hfill$\displaystyle#1$\fi\ignorespaces}

\newcommand{\dotminus}{\mathbin{\text{\@dotminus}}}

\newcommand{\@dotminus}{%
  \ooalign{\hidewidth\raise1ex\hbox{.}\hidewidth\cr$\m@th-$\cr}%
}

\allowdisplaybreaks

\title{\LARGE\textbf{A Generalized Sinkhorn Algorithm for Mean-Field Schr\"{o}dinger Bridge
}}

\author{Asmaa Eldesoukey, Yongxin Chen, Abhishek Halder
\thanks{Asmaa Eldesoukey and Abhishek Halder are with the Department of Aerospace Engineering, Iowa State University, Ames, IA 50011, USA, {\tt\footnotesize{\{{asmaae},ahalder\}@iastate.edu}}.}
\thanks{Yongxin Chen is with the School of Aerospace Engineering, Georgia Institute of Technology, Atlanta, GA 30332, USA, {\tt\footnotesize{ychen3148@gatech.edu}}.}
\thanks{This research was partially supported by NSF awards 2111688, 2450377, 2450378.}
}

\IEEEoverridecommandlockouts
\begin{document}
\bstctlcite{IEEE_b:BSTcontrol}
\maketitle

\begin{abstract}
The mean-field Schr\"odinger bridge (MFSB) problem concerns designing a minimum-effort controller that guides a diffusion process with nonlocal interaction to reach a given distribution from another by a fixed deadline. Unlike the standard Schr\"{o}dinger bridge, the dynamical constraint for MFSB is the mean-field limit of a population of interacting agents with controls. It serves as a natural model for large-scale multi-agent systems. The MFSB is computationally challenging because the nonlocal interaction makes the problem nonconvex. We propose a generalization of the Hopf-Cole transform for MFSB and, building on it, design a Sinkhorn-type recursive algorithm to solve the associated system of integro-PDEs. We analyze the local convergence of the proposed algorithm. We also present numerical examples with repulsive and attractive interactions to illustrate the theoretical contributions.
\end{abstract}

\section{Introduction}\label{sec:Intro}
\noindent\textbf{The problem.} Consider a collection of $N$ interacting agents, with $N$ large, following the controlled noisy dynamics:
\begin{align}
\differential X_t^{i} = \big\{\!\sigma a_{t}^{i} -\dfrac{1}{N}\sum_{j=1}^{N}\nabla W\!\big(X_{t}^{i} - X_{t}^j\big)\big\}\:\differential t + \sigma\:\differential B_{t}^{i},
\label{AgentSDE}
\end{align}
where $i\in\llbracket N\rrbracket := \{1,\hdots,N\}$, time $t\in[0,1]$, and $X_{t}^{i},a_{t}^{i},B_{t}^{i}\in \mathbb{R}^{d}$ denote the state, control action, and (standard Brownian) process noise for the $i$th agent, respectively. In \eqref{AgentSDE}, the constant $\sigma>0$ is the input/noise strength, $\nabla$ denotes the spatial gradient operator, and $W(\cdot)$ is a known \emph{interaction potential} satisfying the assumptions \ref{A1}-\ref{A3}:
\begin{enumerate}[label=\textbf{A\arabic*}]
\item $W \in \mathscr C^2(\mathbb R^d;\mathbb R)$, \label{A1}
\item $W$ is symmetric, i.e., $W(x)=W(-x)\:\forall x\in\mathbb{R}^{d}$, \label{A2}
\item the Hessian, $\mathrm{Hess}(W)$, is uniformly upper bounded. \label{A3}
\end{enumerate}

In the large-population limit ($N\rightarrow\infty$), we are interested in designing a \emph{minimum-effort feedback control strategy}
$a_{t}^{i} =: u_{t}\left(X_t^{i}\right), i\in\llbracket N\rrbracket,$
that guides the initial stochastic state $X_{0}^{i}\sim p_{\mathrm{in}}$ to a final stochastic state $X_{1}^{i}\sim p_{\mathrm{fin}}$ over the fixed time horizon $[0,1]$. Here, $p_{\mathrm{in}},p_{\mathrm{fin}}$ are given initial and final probability distributions that are assumed to be absolutely continuous and of finite second moments. Throughout, we use the same symbols to denote absolutely continuous probability measures and their probability density functions (PDFs). The minimum-effort objective is the average quadratic control action 
$\mathbb{E}\int_{0}^{1} (1/2N)\sum_{i=1}^{N} \vert a_t^{i}\vert^2 \differential t$
where $\mathbb{E}$ denotes the expectation operator induced by the underlying probability distribution, and $\vert\cdot\vert$ the Euclidean norm.

We refer to $N\rightarrow\infty$ as the \emph{mean-field limit} since the collective dynamics are then governed by the distribution
$p_t \approx (1/N)\sum_{i=1}^{N}\delta_{X_t^{i}}$ where $\delta_{X_t^{i}}$ denotes Dirac measure at $X_t^{i}.$
In that limit, $-(1/N)\sum_{j=1}^{N}\nabla W(X_{t}^{i} - X_{t}^j) \approx -(\nabla W * p_t)(X_t^i)$ where $*$ denotes convolution, and the PDF $p_t$ satisfies the \emph{McKean-Vlasov integro-PDE} \cite{mckean1966class}:
\begin{align}\label{MFSB_dynconstr}
\partial_t p_t
    +\nabla\cdot\Big(
      p_t\big(\sigma u_t-\nabla W * p_t\big)\Big)= \frac{\sigma^2}{2}  \Delta p_t,
\end{align}
where $\nabla\cdot$ denotes the divergence, and $\Delta$ is the standard Laplacian operator.

Our stochastic optimal control problem of interest is the following. Given $p_{\mathrm{in}},p_{\mathrm{fin}}$ supported on subsets of $\mathbb{R}^{d}$, determine i) a finite-energy control policy $u: [0,1] \times \mathbb R^d \to \mathbb R^d$, $u: (t,x) \mapsto u_t(x)$, and ii) the PDF-valued curve $p=(p_{t})_{t\in[0,1]}$, with $p_t$ having finite second moments such that $(u,p)$ solves the variational problem:
\begin{subequations}
\begin{align}
&\underset{u,p}{\text{minimize}}\quad\frac{1}{2} \int_0^1 \int_{\mathbb R^d} \vert u_t(x)\vert^2\: p_t(x)\: \differential x \:\differential t, 
\label{MFSB_obj}\\
&\text{subject to }  \quad p_{0} = p_{\mathrm{in}}, \quad 
   \quad  p_{1} = p_{\mathrm{fin}},\label{MFSB_bc}
    \end{align}
    and \eqref{MFSB_dynconstr}.
\label{MFSB_problem}
\end{subequations}%
We refer to \eqref{MFSB_problem} as the \emph{mean-field Schr\"{o}dinger bridge (MFSB)} problem.

\noindent\textbf{Related works.} The MFSB differs from the classical Schr\"{o}dinger bridge (SB) in that the latter is the $W\equiv 0$ special case of \eqref{MFSB_problem}. Note in particular that for $W\neq 0$, the dynamical constraint \eqref{MFSB_dynconstr} is nonlinear in $p_t$. In the absence of interaction, several control-theoretic works have studied the SB with more general drift and diffusion coefficients \cite{chen2015optimal,chen2021stochastic,caluya2020finite,caluya2021wasserstein,nodozi2023neural,teter2025hopf}, with additional state cost \cite{8249875,11239424,teter2025markov,teter2025probabilistic} and constraints \cite{caluya2021reflected,eldesoukey2024schrodinger,eldesoukey2024excursion,movilla2024inferring}. In comparison, the nonconvex problem \eqref{MFSB_problem} is much less explored. 

Backhoff et al. \cite{backhoff2020mean} studied the MFSB problem from a large-deviation perspective and proved existence of solutions under the stated assumptions \ref{A1}-\ref{A3}. They also showed that the large-deviation formalism is equivalent to the dynamic formulation given in \eqref{MFSB_problem}. However, their work did not address numerical solutions.
Chen \cite{chen2023density} considered the static large-deviation problem equivalent to \eqref{MFSB_problem} and proposed to solve its discrete version as a non-standard multi-marginal SB. The reformulated problem in \cite{chen2023density} remains nonconvex and is solved using a proximal gradient algorithm that is guaranteed to converge locally at a sublinear rate. 


\noindent\textbf{Motivation.} 
Problem \eqref{MFSB_problem} has natural applications in steering multi-agent populations as in crowd dynamics \cite{colombo2012nonlocal}, population biology \cite{topaz2006nonlocal,zhang2025deciphering}, and in swarm robotics \cite{elamvazhuthi2020mean}. For instance, $W$ may model limited communication or information sharing due to trust, privacy, or power constraints.    

\noindent\textbf{Contributions.} Our technical contributions are threefold.
\begin{itemize}
\item We propose a generalized Hopf-Cole transform to derive a new Schr\"{o}dinger system for \eqref{MFSB_problem} (Sec. \ref{subsec:GeneralizedHopfCole}).

\item Using the above, we design a generalized Sinkhorn algorithm (Sec. \ref{subsec:GeneralizedSinkhorn}) to solve \eqref{MFSB_problem} numerically. Beyond novelty, the proposed algorithm is a significant structural generalization of existing Sinkhorn algorithms, and avoids tensorial optimization as in \cite{chen2023density}.

\item We analyze the local convergence of the proposed algorithm (Sec. \ref{subsec:GeneralizedSinkhorn}), and illustrate its effectiveness via numerical examples (Sec. \ref{sec:NumericalResults}).
\end{itemize}


\section{Conditions for Optimality}\label{sec:ConditionsForOptimality}
For computational convenience, we scale the interaction potential as $W\mapsto \sigma^2 W$
\emph{without loss of generality}. Therefore, we consider the scaled version of \eqref{MFSB_dynconstr}:
\begin{align}\tag{\ref{MFSB_dynconstr}$^{\prime}$} 
\label{MFSB_dynconstr-W-scaled}
\partial_t p_t
    +\nabla\!\cdot\!(
      p_t(\sigma u_t-\sigma^2(\nabla W * p_t))) = (\sigma^2/2) \Delta p_t.
\end{align}
Backhoff et al. \cite[Proposition 1.1]{backhoff2020mean} formulated the MFSB problem as a large-deviation problem and proved that it admits at least one solution. They also derived the equivalent optimal control formulation in \eqref{MFSB_problem}, cf. \cite[Lemma~1.1]{backhoff2020mean}. 
We summarize their existence result for the dynamic problem, adapted to the scaling $W\mapsto \sigma^2 W$, as follows.
\begin{proposition}\label{prop:DualHJBsystem}
\cite[Theorem~1.3, Corollaries~1.1 and 1.2]{backhoff2020mean} Under \ref{A1}-\ref{A3} and the assumptions on $p_{\rm in}$ and $p_{\rm fin}$, Problem~\eqref{MFSB_problem} admits at least one solution $(u^{\rm opt}, p^{\rm opt})$.  Moreover,  $u^{\rm opt} \in H_{-1}\big((p_t^{\rm opt})_{t\in [0,1]}\big)$, where $H_{-1}\big((p_t^{\rm opt})_{t\in [0,1]}\big)$ denotes the closure of the space of smooth gradient vector fields in the  $p_t^{\rm opt}$-weighted $L^2$ norm. If, in addition, $p^{\rm opt} \in \mathscr C^{1,2}([0,1] \times \mathbb R^d; \mathbb R)$, $u^{\rm opt}  \in \mathscr C^{1,2}([0,1] \times \mathbb R^d; \mathbb R^d)$, and $p^{\rm opt}>0$, then there exist $\psi, \phi \in \mathscr C^{1,2}([0,1] \times \mathbb R^d; \mathbb R)$ such that 
\begin{align} \label{eq:optimal-u}
    u_t^{\rm opt}(x) &= \sigma \nabla \psi_t(x),\\
\label{eq:fixedpoint-p}
    p_t^{\rm opt}(x) &= \exp(-2 (W*p_t^{\rm opt})(x)+\psi_t(x)+ \phi_t(x)),
\end{align}
where $\psi_t(x) = \psi(t,x)$ and $\phi_t(x)=\phi(t,x)$. Additionally, $\psi_t, \phi_t$ satisfy the Hamilton-Jacobi-Bellman (HJB) equations
\begin{align}\label{eq:HJB-for}
&\partial_t \psi_t(x) + \frac{\sigma^2}{2}  \vert \nabla \psi_t(x) \vert^2 +  \frac{\sigma^2}{2}  \Delta \psi_t(x) = \nonumber \\
&\sigma^2\!\int_{\mathbb R^d}\!\!  \langle\nabla \psi_t(x)-\nabla \psi_t(y),\nabla W(x-y)\rangle p_t^{\rm opt}(y)\, \dy,\\
    &-\partial_t \phi_t(x) +  \frac{\sigma^2}{2} \vert \nabla \phi_t(x)\vert^2 + \frac{\sigma^2}{2}  \Delta \phi_t(x) = \nonumber \\ 
    &  \sigma^2 \!\int_{\mathbb R^d}\!\!\langle\nabla \phi_t(x) - \nabla \phi_t(y),\nabla W(x-y)\rangle p_t^{\rm opt}(y) \, \dy,\label{eq:HJB-back}
\end{align}
where  $\big \langle\cdot,\cdot\rangle$ is the Euclidean inner product, subject to the boundary conditions 
\begin{subequations}
\begin{align}
p_{\rm in}(x) &= \exp(-2 (W*p_{\rm in})(x)+\psi_0(x)+ \phi_0(x)),\\
p_{\rm fin}(x) &= \exp(-2 (W*p_{\rm fin})(x)+\psi_1(x)+ \phi_1(x)).
\end{align} \label{BCfactorization}
\end{subequations}
\end{proposition}
Eq.~\eqref{eq:fixedpoint-p} shows $p_t^{\rm opt}$ as a product of positive factors, analogous to the non-interacting SB solutions. Next, we use this to obtain a generalized Schr\"odinger system in a positive tuple $(p_t^{\rm opt}, \varphi_t, \hat \varphi_t)$ for the interacting setting.


\section{Schr\"{o}dinger System and its Solution}\label{sec:SchrodingerSystemAndItsSolution}
\subsection{Hopf-Cole Transform}\label{subsec:GeneralizedHopfCole}
We now transcribe the coupled system of equations~\eqref{eq:fixedpoint-p}-\eqref{BCfactorization} into a \emph{Schr\"{o}dinger system}, for which we will design a generalized Sinkhorn recursion for numerical computation. To this end, we use the \emph{Hopf-Cole transform} \cite{hopf1950partial,cole1951quasi}:
\begin{align}
(\psi_t,\phi_t)\mapsto(\varphi_t,\hat\varphi_t):=(\exp \psi_t, \exp \phi_t).
\label{defHopfCole}    
\end{align}
Combining \eqref{eq:fixedpoint-p} and \eqref{defHopfCole}, we get
\begin{subequations}\label{eq:Sch-sys}
\begin{align}\label{eq:update-p_t}
    p_t^{\rm opt}(x)  &= e^{-2 (W*p_t^{\rm opt})(x)} \varphi_t(x) \hat \varphi_t(x).
    \end{align}
Let  $b_t^{\rm opt}:=  - (\nabla W * p_t^{\rm opt})$, and for $\xi_t>0$, let
\begin{align*}
 \!\! \!\!\!  Q_t(\xi_t, p_t)(x)\!:= -\!\!\int_{\mathbb R^d}\!\! \big \langle \nabla\log \xi_t(y),\!\nabla W(x-y)\big\rangle \, p_t(y)\dy.
\end{align*}
Then, by \eqref{eq:HJB-for},\eqref{defHopfCole} and the identity $\vert \nabla (\log \varphi_t) \vert^2 +  \Delta(\log \varphi_t) = \Delta \varphi_t/\varphi_t$, one can verify that $\varphi_t$ satisfies
\begin{align}
 \partial_t \varphi_t \! + \! \sigma^2 \langle b_t^{\rm opt}  ,  \nabla \varphi_t\rangle \!+\! \frac{\sigma^2}{2}  \Delta \varphi_t \! = \! \sigma^2 Q_t(\varphi_t,p_t^{\rm opt}) \varphi_t.\label{eq:backward-MV}
\end{align}
Analogously, from \eqref{eq:HJB-back} and \eqref{defHopfCole}, $\hat \varphi_t$ satisfies
\begin{align}\label{eq:forward-MV}
 \!\! \partial_t \hat\varphi_t \!-\!\sigma^2  \langle b_t^{\rm opt} , \nabla \hat \varphi_t \rangle \! -\! \frac{\sigma^2}{2}  \Delta \hat\varphi_t \!= \! - \sigma^2 Q_t(\hat \varphi_t, p_t^{\rm opt}) \hat \varphi_t.
\end{align}
Eqs.~\eqref{eq:update-p_t}-\eqref{eq:forward-MV} are subject to
 \begin{align}
      e^{-2 (W*p_{\rm in})(x)}\varphi_0(x) \hat \varphi_0(x) = p_{\rm in}(x), \\ 
      e^{-2 (W*p_{\rm fin})(x)}\varphi_1(x) \hat \varphi_1(x) = p_{\rm fin}(x). \label{eq:BC-final}
 \end{align}
\end{subequations}

The system of equations \eqref{eq:Sch-sys} constitutes a nonlinear, coupled-through-time \emph{Schr\"odinger system with mean-field interactions}. As expected, for $W =0$, \eqref{eq:Sch-sys} reduces to the classical Schr\"odinger system. To numerically solve \eqref{eq:Sch-sys}, we next propose a generalized Sinkhorn algorithm that handles the nonlinearities in the drift and source/sink (reaction) terms in \eqref{eq:backward-MV}-\eqref{eq:forward-MV}.


\begin{algorithm}[t!] 
\setlength{\intextsep}{2pt}
    \caption{Mean-Field Sinkhorn}
    \label{alg}
     \textbf{Inputs} $p_{\rm in}$, $p_{\rm fin}$, $W$, $\sigma$, $p^{(0)}$ with $p_0^{(0)} = p_{\rm in}$ and $p_1^{(0)} = p_{\rm fin}$, $\varphi^{(0)},\hat\varphi^{(0)}$, $N_1$, $N_2, N_3$, \texttt{tol}.\\
     \textbf{Outputs} $p,\varphi,\hat\varphi$.
    \begin{enumerate}
   \item \textbf{for} $k = 0, \, \cdots, {N_1}$ 
   \item Compute $b^{(k)} = -  \nabla W*p^{(k)}$
        \item Set $\varphi^0 \leftarrow  \varphi^{(k)}$ and $\hat \varphi^0 \leftarrow  \hat \varphi^{(k)}$.
      \item \textbf{for} $j = 0, \, \cdots, {N_2}$ 
      \item Compute $Q_t(\varphi_t^j,p_t^{(k)})$ and $Q_t(\hat\varphi_t^j,p_t^{(k)})$
       \item Set $\varphi_1^{\{0\}} \leftarrow  \varphi_1^j$.
    \item \textbf{for} $i = 0, \, \cdots, {N_3}$
     \item Integrate backward in time
    \begin{subequations}
     \begin{align}
        \partial_t \varphi_t^{\{i+1\}} +\sigma^2 \langle b_t^{(k)}  &, \nabla \varphi_t^{\{i+1\}} \rangle  + \frac{\sigma^2}{2}  \Delta \varphi_t^{\{i+1\}}  \nonumber \\
        & = \sigma^2 Q_t(\varphi_t^j,p_t^{(k)}) \varphi_t^{\{i+1\}},  \label{eq:backward-alg}
     \end{align}
     with  $\varphi_{t=1}^{\{i+1\}}(x) =\varphi_1^{\{i\}}(x)$.
\item Compute
 \begin{align}\label{eq:init-BC-alg}
     \hat \varphi_0^{\{i+1\}} (x) =  \frac{p_{\rm in}(x)}{\varphi_{0}^{\{i+1\}}(x)} e^{2(W*p_{\rm in})(x)}.
 \end{align}
\item  Integrate forward in time
              \begin{align}\label{eq:forward-alg}
           \partial_t \hat\varphi_t^{\{i+1\}} -\sigma^2  \langle b_t^{(k)} &, \nabla  \hat\varphi_t^{\{i+1\}}\rangle   -\frac{\sigma^2}{2}  \Delta \hat\varphi_t^{\{i+1\}}  \nonumber \\
           & = -\sigma^2 Q_t(\hat\varphi_t^j,p_t^{(k)}) \hat\varphi_t^{\{i+1\}}, 
    \end{align}
    with $\hat \varphi_{t=0}^{\{i+1\}}(x) = \hat \varphi_0^{\{i+1\}}(x)$.
 \item    Compute
     \begin{align}\label{eq:final-BC-alg}
           \varphi_1^{\{i+1\}} (x) =  \frac{p_{\rm fin}(x)}{\hat\varphi_{1}^{\{i+1\}}(x)} e^{2(W*p_{\rm fin})(x)}.
     \end{align}
      \end{subequations}
  \item   \textbf{if} $d_H(\varphi_1^{\{i\}},\varphi_1^{\{i+1\}}) <\texttt{tol}$, Break; \textbf{end if}
   \item \textbf{end for}
   \item Set $ \varphi^{j+1} \leftarrow  \varphi^{\{i+1\}}, \hat\varphi^{j+1} \leftarrow \hat\varphi^{\{i+1\}}$
   \item  \textbf{if} $d_H^\oplus\big((\varphi^{j+1}, \hat \varphi^{j+1}), (\varphi^j, \hat \varphi^j)\big) \! <\!\texttt{tol}$, Break;  \textbf{end if}
   \item \textbf{end for}
 \item     Set $\varphi^{(k+1)} \leftarrow \varphi^{j+1}$ and $\hat\varphi^{(k+1)} \leftarrow \hat\varphi^{j+1}$.
  \item   Compute $\alpha_t (p^{(k)}) = \int e^{-2(W*p_t^{(k)})} \varphi_t^{(k+1)} \hat \varphi_t^{(k+1)} \,\dx$,
      \begin{align} \label{eq:iteration-nl}
      \!\! \!\!\!\!\! \!\!\!\!   p_t^{(k+1)}  = \frac{1}{\alpha_t(p^{(k)})} e^{-2(W*p_t^{(k)})} \varphi_t^{(k+1)} \hat \varphi_t^{(k+1)}. 
     \end{align}
      \item   \textbf{if} $d_H^\mathcal{S}(p^{(k+1)},p^{(k)}) <\texttt{tol}$, Break; \textbf{end if}
      \item \textbf{end for}
      \end{enumerate} 
\end{algorithm}

\subsection{Generalized Sinkhorn Algorithm}\label{subsec:GeneralizedSinkhorn}
We introduce Algorithm~\ref{alg}, which consists of nested fixed-point iterations. We emphasize that notation such as $p = (p_t)_{t\in [0,1]}$ represents trajectories/curves over time.
The initial guess $(p^{(0)}, \varphi^{(0)}, \hat{\varphi}^{(0)})$ of Algorithm~\ref{alg} can be taken as the solution of the classical (non-interacting) SB problem.

The innermost Sinkhorn-type iteration, with index $i$, solves \eqref{eq:backward-MV}-\eqref{eq:BC-final} with fixed $p_t = p_t^{(k)}$ and fixed reaction coefficients $Q_t(\varphi_t=\varphi_t^j,p_t^{(k)}),Q_t(\hat\varphi_t =\hat \varphi_t^{j},p_t^{(k)})$. Hence, this step consists of time-marching the solution of the backward-forward linear Kolmogorov equations. The terms ``backward" and ``forward" refer to the time direction in which the corresponding PDE is well-posed according to the sign of the diffusion term $\Delta(\cdot)$.
Upon completing the innermost iteration, we obtain trajectories $(\varphi,\hat \varphi)$ that are used to compute the next reaction coefficients for the same $p_t^{(k)}$. 
The outer iteration in Algorithm~\ref{alg} updates $p_t^{(k)}$ using \eqref{eq:iteration-nl}. At its convergence, $p^{(k+1)} = p^{(k)}$ in the outer fixed-point metric. This resulting $p^{(k+1)}$ is a PDF that is consistent with the prescribed marginals and satisfies the fixed-point relation $\eqref{eq:iteration-nl}$, with $\varphi, \hat \varphi$ solving the backward-forward PDEs in which $p^{(k+1)}$ appears in the drift and reaction terms. Upon the algorithm convergence, one must verify that $\alpha_t(p^{(k)})=1$ for all $t \in [0,1]$, or alternatively, add $\vert \alpha_t(p^{(k)})-1 \vert \approx 0 \ \forall t\in [0,1]$ as a stopping criterion, to ensure that \eqref{eq:fixedpoint-p} is satisfied at the algorithm's fixed point. When the condition on $\alpha_t$ holds, the output of the algorithm gives the tuple $(p_t,\varphi_t, \hat \varphi_t)$ that satisfies the Schr\"odinger system \eqref{eq:Sch-sys}.
We explain the termination criteria involving the distances $d_H(\cdot,\cdot),d_H^\oplus(\cdot, \cdot)$, and $d_H^{\mathcal S}(\cdot,\cdot)$ in Section \ref{subsec:ConvergenceThmProof}.

\subsection{Convergence Analysis}\label{subsec:ConvergenceThmProof}
Here, we provide sufficient conditions for the convergence of Algorithm~\ref{alg}. We consider a compact set $\Omega \subset \mathbb R^d$ and the Banach space $L^\infty(\Omega)$ with norm
$\Vert \cdot \Vert_\infty := \sup_{x \in \Omega} \vert \cdot \vert$. Throughout, $\sup$ or $\inf$ are understood as ${\rm ess} \sup$ and ${\rm ess} \inf$. The set $\Omega$ is considered the spatial domain over which Algorithm~\ref{alg} is applied. Let $\mathcal K := \{ f \in L^\infty(\Omega): f\geq  0 \text{ a.e.}\}$ be the cone of non-negative functions in $L^\infty(\Omega)$, and let $\mathcal K_0$ be the interior of $\mathcal K$. The \emph{Hilbert metric} $d_H$ \cite{bushell1973hilbert,georgiou2015positive} between $f,g \in \mathcal K_0$ is
\begin{align}\label{eq:Hilbert-metric-def}
 \!\! \!\!  d_H(f,g) := \log \sup_{x \in  \Omega}\frac{f(x)}{g(x)}- \log \inf_{x \in \Omega} \frac{f(x)}{g(x)}.
\end{align}
The Hilbert metric is \emph{projective} because $d_H(f, g) =0$ when $f = cg$ for any $c > 0$. 

We say a positive map $\mathcal C: \mathcal K_0 \to \mathcal K_0$ is \emph{locally contractive with respect to the Hilbert metric} near $f \in \mathcal K_0$ if there exists a neighborhood $D \subset \mathcal K_0$  containing  $f$ and a constant $\lambda\in[0,1)$ such that
\begin{align*}
    d_H(\mathcal C(g), \mathcal C(h)) \leq \lambda\, d_H(g,h) \quad \forall g,h \in D.
\end{align*}
Whenever the previous inequality holds for all $g, h \in \mathcal K_0$, the map $\mathcal C$ is \emph{globally contractive}.

For fixed drift and reaction coefficients, the innermost loop of Algorithm~\ref{alg} is a composition of the two positive linear maps: $\varphi_1^{\{i+1\}} \mapsto \varphi_0^{\{i+1\}}$, $\hat \varphi_0^{\{i+1\}} \mapsto \hat\varphi_1^{\{i+1\}}$, defined by the Kolmogorov propagators, and the two scaled inversions in \eqref{eq:init-BC-alg},\eqref{eq:final-BC-alg}. The scaled inversions are isometries, while the Kolmogorov propagators are positive with bounded kernels; hence, by the Birkhoff-Bushell theorem \cite[Theorem 3.1]{bushell1973hilbert}, they are globally contractive, all with respect to the Hilbert metric. Accordingly, their composition is globally contractive, and iterates of this composition converge globally in $d_H(\cdot,\cdot)$ to a unique fixed point. This aligns with contraction analysis for the classical Sinkhorn; see \cite{chen2016entropic}.
Thus, it remains to investigate the convergence of the loops indexed by $j,k$.

To this end, we consider the space of positive trajectories
$\mathcal A := \{f = (f_t)_{t\in[0,1]}: f_t \in  \mathcal K_0 \ \forall t \in [0,1] \},$
and the space of positive probability density trajectories
$\mathcal S := \{ h = (h_t)_{t\in[0,1]}: h_t \in \mathcal K_0, \int_\Omega h_t(x) \, \dx =1 \ \forall t \in [0,1] \},$
with $\mathcal S \subset \mathcal A$. Here and below, we use $\sup_{t}$ to denote $\sup_{t\in [0,1]}$, for brevity.
We let 
\begin{align}
    d_H^{\mathcal S} (p,q) := \sup_t d_H(p_t,q_t) \quad\forall p, q \in \mathcal S.
\end{align}
For $\boldsymbol{\varphi} = (\varphi,\hat\varphi), \boldsymbol{\mu} = (\mu,\hat\mu) \in \mathcal A \times \mathcal A$, we define 
\begin{align}\label{eq:dH-prod}
  \!\!\!  d_H^{\oplus}  (\boldsymbol{\varphi} ,\boldsymbol{\mu}) \!  := \! \max \big\{ \sup_{t}  d_H(\varphi_t,\mu_t) ,  \sup_t  d_H(\hat \varphi_t, \hat \mu_t)  \big \}.
\end{align}
We construct the map $\mathcal C: \mathcal S \to  \mathcal S$, expressed as
\begin{align}
    (\mathcal{C}(p))_t &:= \frac{1}{\alpha_t(p)} e^{-2(W*p_t)} \, \varphi_t^\star(p) \, \hat \varphi_t^\star(p) \quad \forall t \in [0,1],\label{eq:map c}
\end{align}
where $\varphi^\star(p), \hat \varphi^\star(p)$ are the limits of the intermediate loop for a fixed $p$ and $\alpha_t(p) = \int_\Omega e^{-2(W*p_t)} \, \varphi_t^\star(p) \, \hat \varphi_t^\star(p) \, \dx$. For a fixed $p \in \mathcal S$, we define $\mathcal G_p: \mathcal A \times \mathcal A \to \mathcal A \times \mathcal A$ by
\begin{align} \label{eq:map g}
    \mathcal G_p(\boldsymbol{\varphi}^{j}(p))= \boldsymbol{\varphi}^{j+1}(p), \quad \boldsymbol{\varphi}^{j}(p)= (\varphi^j(p), \hat \varphi^j(p)).
\end{align}
 We next study the local convergence of the iterations induced by $\mathcal{C}$, i.e., the outer loop in Algorithm~\ref{alg} (Theorem~\ref{thm:k-iter}) with respect to $d_H^{\mathcal S}(\cdot, \cdot)$ and by $\mathcal G_{p}$, i.e., the intermediate loop, with respect to $d_H^\oplus(\cdot, \cdot)$ (Theorem~\ref{thm:j-iter}), and then the overall local convergence for Algorithm~\ref{alg} (Theorem~\ref{thm:overall}).

\begin{theorem}\label{thm:k-iter}
Let $p^\star$ be a fixed point of the map $\mathcal C$ defined in \eqref{eq:map c}. For $r>0$, consider the ball $B_r(p^\star) := \{ p \in \mathcal S: d_H^{\mathcal S}(p,p^\star) \leq r\}.$
In addition to \ref{A1}-\ref{A3}, we make the following assumptions \ref{Assump.W}-\ref{Assump.q_t} on all $p,q \in B_r(p^\star)$ for some nonnegative constants $a_1(r),a_2(r),a_3(r),c_1(r),c_2(r)<\infty$.
    \begin{enumerate}[label=\textbf{B\arabic*}]
    \item   $W,\nabla W, \Delta W$ are  in $L^\infty$. \label{Assump.W}
    \item $\sup_{t}  \Vert \nabla \psi_t^\star(p) \Vert_\infty  \leq a_1$ and $\sup_{t}  \Vert \nabla \phi_t^\star(p) \Vert_\infty  \leq a_2$.\label{Assump.psi_t}
    \item $\Vert\psi_1^\star(p)-\psi^\star_1(q) \Vert_\infty \leq c_1 \sup_{t} \Vert  p_t - q_t\Vert_1$ and $\Vert \phi_0^\star(p)-\phi_0^\star(q)\Vert_\infty \leq c_2 \sup_{t}\Vert  p_t - q_t \Vert_1$. \label{Assump.phipsi-boundary}
    \item $\sup_{t}  \Vert\nabla p_t \Vert_1 < a_3$. \label{Assump.q_t}
    \end{enumerate} 
Here $\psi_t^\star (\cdot)= \log \varphi_t^\star(\cdot)$  and $\phi_t^\star (\cdot)= \log \hat\varphi_t^\star(\cdot)$. Let $W = \beta \overline{W}$ for some scaling $\beta >0$, and let the constant
     \begin{align*}
        \lambda :=  2e^{2r+1} \big[\frac{2\beta \Vert \overline{W} \Vert_{\infty}}{e} \! +\! \frac{2\sigma^2 (a_1\!+\!a_2) \beta \Vert \nabla \overline{W} \Vert_{\infty} \!+\! c_1 \!+ \!c_2}{1-e \sigma^2 \beta Z} \big],
     \end{align*}
   for $Z:= \Vert \Delta \overline{W}\Vert_\infty + a_3 \Vert \nabla \overline{W}\Vert_\infty$ with $\sigma^2 \beta Z<e^{-1}$.
     If $\lambda<1$, then the map $\mathcal C$ is locally contractive near $p^\star$ such that $d_H^{\mathcal S} \big(\mathcal C(p), \mathcal C(q)\big) \leq \lambda d_H^{\mathcal S}(p,q) \ \forall p,q \in B_r(p^\star)$.
Additionally, for every $p^{(0)} \in B_r(p^\star)$, the iteration $p^{(k+1)} = \mathcal C(p^{(k)}) \to p^\star$ as $k \to \infty$, and $p^\star$ is the unique fixed point in $B_r(p^\star)$. 
\end{theorem}

\begin{proof}
Letting $\mathcal E_t(p) := e^{-2(W*p_t)}$ $\forall t \in [0,1]$, we write
$(\mathcal C(p))_t = \alpha_t(p)^{-1}\mathcal E_t(p) \, \varphi_t^\star(p) \, \hat \varphi_t^\star(p).$
From the definition in \eqref{eq:Hilbert-metric-def}, one can verify that $d_H(p_1 p_2, q_1q_2) \leq d_H(p_1,q_1) + d_H( p_2, q_2)$ for all $p_{i},q_i \in \mathcal K_0, i \in\{1,2\}$. Accordingly,
\begin{align}\label{eq:Big-Inequality}
   & d_H^{\mathcal S} \big(\mathcal C(p),  \mathcal C(q)\big) \leq \sup_{t} d_H\big(\mathcal E_t(p), \mathcal E_t(q)\big) \nonumber \\
&+\sup_{t} d_H\big(\varphi_t^\star(p), \varphi_t^\star(q)\big) \! + \!\sup_{t} d_H\big(\hat\varphi_t^\star(p), \hat \varphi_t^\star(q)\big).
\end{align}
We now investigate the terms on the right-hand side of \eqref{eq:Big-Inequality} individually to obtain their respective upper bounds in terms of $d_H^{\mathcal S}(p,q)$.

For all $f,g \in \mathcal K_0$, two pertinent inequalities are
\begin{align}
     d_{H} (f,g) 
     &\leq 2 \Vert \log  f - \log g \Vert_\infty,         \label{eq:infty-H}  \\
    \Vert f -g \Vert_1 &\leq  e^{d_H(f,g)} -1 , \quad \text{when  } \Vert f \Vert_1 = \Vert g \Vert_1=1, \label{eq:1-H}
\end{align}
cf. Corollary 1 \cite{birkhoff1957extensions} and Lemma 2.2 \cite{eckstein2025hilbert}. 

Using \eqref{eq:infty-H}, we have that $d_H\big(\mathcal E_t(p), \mathcal E_t(q)\big)  \leq 4 \Vert \beta \big(\overline{W}*(p_t-q_t)\big) \Vert_\infty$. By Young's convolution inequality \cite[Proposition 2.33]{van2022functional} and \ref{Assump.W}, we get $d_H\big(\mathcal E_t(p), \mathcal E_t(q)\big) \leq 4 \beta \Vert \overline{W} \Vert_{\infty}  \Vert p_t-q_t \Vert_1.$
Then, by \eqref{eq:1-H}, we obtain
\begin{align*}
   d_H\big(\mathcal E_t(p), \mathcal E_t(q)\big)  \leq 4 \beta \Vert \overline{W} \Vert_{\infty} \big(e^{d_H(p_t,q_t)} -1 \big).
\end{align*}
As $p,q \in B_r(p^\star)$, then $d_H(p_t,q_t) \leq 2 r$, and \footnote{$e^x-1 \leq x e^x$ for $x \geq 0$.}it follows that $e^{d_H(p_t,q_t)} -1 \leq e^{2r} d_H(p_t,q_t)$. Thus, 
\begin{align}
   \!\!\sup_{t} d_H\big(\mathcal E_t(p), \mathcal E_t(q)\big)  &\leq 4 \beta \Vert \overline{W} \Vert_{\infty} e^{2r} d_H^{\mathcal S}(p,q). \label{eq:bound-E}
\end{align}

Let $\gamma_t := \log \varphi_t^\star(p) - \log \varphi_t^\star(q) \equiv \psi_t^\star(p) - \psi_t^\star(q)$. Using \eqref{eq:infty-H}, we have $d_H\big(\varphi_t^\star(p), \varphi_t^\star(q)\big) \leq 2\Vert \gamma_t \Vert_\infty$, and so
\begin{align}
    \sup_{t} d_H\big(\varphi_t^\star(p), \varphi_t^\star(q)\big) \leq 2\sup_{t}\Vert \gamma_t \Vert_\infty.
    \label{FromHilbertToInfinityNormOfgamma}
\end{align}
Next, we derive an upper bound on $\sup_{t}\Vert \gamma_t \Vert_\infty$, which by \eqref{FromHilbertToInfinityNormOfgamma}, will give an upper bound on $\sup_{t} d_H\big(\varphi_t^\star(p), \varphi_t^\star(q)\big)$.

From \eqref{eq:HJB-for} and by direct calculations, we get
\begin{align}
    \partial_t \gamma_t &+ \frac{\sigma^2}{2} \langle\nabla \psi_t^\star(q)  + \nabla \psi_t^\star(p) -  2(\nabla W*p_t),\nabla \gamma_t\rangle \nonumber \\
    &+    \frac{\sigma^2}{2}  \Delta \gamma_t = L_t + M_t + R_t,
\end{align}
where 
$L_t := \sigma^2\langle\nabla \psi_t^\star(q), \nabla W*(p_t-q_t)\rangle,$
$M_t := \sigma^2 \int_{\Omega} \langle \nabla \psi_t^\star(p) (y) , \nabla W(x-y) \rangle (q_t-p_t)(y) \, \dy,$
and
$R_t := \sigma^2 \int_{\Omega}  \gamma_t(y)  \nabla_y \cdot \big(\nabla W(x-y)\,q_t(y) \big)\, \dy.$
We let $\eta_{s} := \gamma_{1-s}$ for $s \in [0,1]$, i.e., $\eta$ is the time reversal of $\gamma$. Then, $\eta_s$ solves the forward parabolic PDE
\begin{align}
        -\partial_s \eta_{s} &+ \frac{\sigma^2}{2} \langle\nabla \psi^\star_{1-s}(q)  + \nabla \psi^\star_{1-s}(p) - 2(\nabla W*p_{1-s}),\nabla \eta_s\rangle \nonumber \\
    & +    \frac{\sigma^2}{2}  \Delta \eta_s = L_{1-s} \!+ \!M_{1-s} +  R_{1-s}. \label{eq:eta}
\end{align}
Here, we invoke \cite[Theorem 2.10]{lieberman1996second}, which establishes 
\begin{align*}
    \sup_{s \in[0,1]} \Vert \eta_s\Vert_\infty \leq e \Big(\sup_{\mathscr P_f} \vert \eta \vert + \!\! \sup_{s \in[0,1]} \!\!\Vert L_{1-s} \! +\! M_{1-s}\!+\!R_{1-s}\Vert_\infty \Big),
\end{align*}
where $e\approx 2.71828$ is the Euler number, and $\mathscr P_f := (\Omega \times \{0\}) \bigcup (\partial \Omega \times(0,1))$, representing the union of spatial and temporal boundaries.
Equivalently,
\begin{align} \label{eq:liebermann}
     \sup_{t} \Vert \gamma_t\Vert_\infty \leq e \Big(\!\sup_{\mathscr P_b} \vert \gamma \vert \!+ \! \sup_{t} \!\Vert L_t \!+ \! M_t \!+ \!R_t\Vert_\infty \Big),
\end{align}
where $\mathscr P_b := (\Omega \times \{1\}) \bigcup (\partial \Omega \times(0,1))$.

If the same Dirichlet spatial boundary conditions are enforced on $\varphi^\star(p),  \varphi^\star(q)$, then $\sup_{x\in \partial \Omega} \vert \psi_t^\star(p) -\psi_t^\star(q) \vert =0$ for all $t \in [0,1]$  Accordingly, only bounds on the difference of $\psi$ at the temporal boundary, namely, $\Vert \psi_1^\star(p)-\psi_1^\star(q) \Vert_\infty$, remain. By Assumption~\ref{Assump.phipsi-boundary}, 
$\sup_{\mathscr P_b} \vert \gamma \vert  =    \Vert \gamma_1 \Vert_\infty \leq c_1 e^{2r} d_H^{\mathcal S}(p,q).$

Since $\Vert L_t \Vert_\infty \leq \sigma^2 \Vert \nabla \psi_t^\star(q) \Vert_\infty  \Vert \nabla W*(p_t-q_t) \Vert_\infty$, by \ref{Assump.psi_t} and Young's convolution inequality, it holds that $\Vert L_t \Vert_\infty \leq \sigma^2 a_1 \beta \Vert \nabla \overline{W} \Vert_{\infty}  \Vert p_t-q_t \Vert_1.$
By virtue of \eqref{eq:1-H},
    $\sup_{t}\Vert L_t \Vert_\infty \leq \sigma^2 a_1 \beta \Vert \nabla \overline{W} \Vert_{\infty}  e^{2r} d_H^{\mathcal S}(p,q).$
Likewise,
     $\sup_{t}\Vert M_t \Vert_\infty \leq \sigma^2 a_1 \beta \Vert \nabla \overline{W} \Vert_{\infty}  e^{2r} d_H^{\mathcal S}(p,q).$
Further, by \ref{Assump.q_t}, 
     $\sup_{t}\Vert R_t \Vert_\infty \leq \sigma^2  \beta Z \sup_{t} \Vert \gamma_t\Vert_\infty,$
for $Z =  \Vert \Delta \overline{W}\Vert_\infty + a_3 \Vert \nabla \overline{W}\Vert_\infty$.
By the triangle inequality, $\Vert L_{t} \!+ \!M_{t} + R_t\Vert_\infty \leq  \Vert L_{t} \Vert_\infty +\Vert M_{t} \Vert_\infty + \Vert R_t \Vert_\infty$. Hence, substituting the bounds on $\sup_{\mathscr P_b} \vert \gamma \vert$, $\sup_t \Vert L_t \Vert_\infty$, $\sup_t \Vert M_t\Vert_\infty$, and $\sup_t \Vert R_t\Vert_\infty$ into the inequality \eqref{eq:liebermann} gives the sought-after bound on $\sup_{t}\Vert \gamma_t\Vert_\infty$. Then, by \eqref{FromHilbertToInfinityNormOfgamma}, we establish that 
\begin{align}\label{eq:Bound-varphi}
   &\sup_{t} d_H\big(\varphi_t^\star(p), \varphi_t^\star(q)\big) \leq  \nonumber \\
   & \frac{2 e^{2r+1}}{1-e \sigma^2 \beta Z} \big(2\sigma^2 a_1 \beta \Vert \nabla \overline{W} \Vert_{\infty} + c_1\big)  d_H^{\mathcal S}(p,q),
\end{align}
for $\sigma^2 \beta Z<e^{-1}$.
All steps used to obtain the previous relation apply verbatim for $\hat \varphi$ except for the need to reverse the time to invoke \cite[Theorem 2.10]{lieberman1996second}. The result is
\begin{align}\label{eq:Bound-hatvarphi}
   &\sup_{t} d_H\big(\hat \varphi_t^\star(p), \hat\varphi_t^\star(q)\big) \leq  \nonumber \\
   &  \frac{2 e^{2r+1}}{1-e \sigma^2 \beta Z} \big(2\sigma^2 a_2 \beta \Vert \nabla \overline{W} \Vert_{\infty} + c_2\big)  d_H^{\mathcal S}(p,q),
\end{align}
Plugging \eqref{eq:bound-E}, \eqref{eq:Bound-varphi}, and \eqref{eq:Bound-hatvarphi} into \eqref{eq:Big-Inequality} yields the constant $\lambda>0$ in the theorem statement.

As long as $p^{(0)} \in B_r(p^\star)$ and $\lambda<1$, $d_H^{\mathcal S} \big(p^{(1)}, p^\star\big) = d_H^{\mathcal S} \big(\mathcal C(p^{(0)}), \mathcal C(p^\star)\big)  \leq \lambda \, d_H^{\mathcal S} (p^{(0)},p^\star) < r,$ and $p^{(1)}$ stays in $B_r(p^\star)$. Hence, $d_H^{\mathcal S} \big( p^{(k)}, p^\star\big) \leq \lambda^k d_H^{\mathcal S}(p^{(0)}, p^\star) \leq \lambda^k r <r.$
Since $\lambda^k \to 0$ as $k \to \infty$, the fixed-point iteration $p^{(k+1)} = \mathcal C(p^{(k)})$ converges to $p^\star$. If $q^\star$ denotes another fixed point of $\mathcal C$ in $B_r(p^\star)$ and $p^{(0)} = q^\star$, all consecutive iterates of $p^{(0)}$ converge to $p^\star$. This necessitates that $q^\star = p^\star$ and the fixed point of $\mathcal C$ is unique in $B_r(p^\star)$.
\end{proof}

\begin{theorem}\label{thm:j-iter}
Let $\boldsymbol{\varphi}^\star(p) = (\varphi^\star(p), \hat \varphi^\star(p))$ be a fixed point of the map $\mathcal G_p$ defined in \eqref{eq:map g} for a fixed $p$. For $\rho :=\rho(p)>0$, define the ball $V_{\rho}^p :=  \{ \boldsymbol{\varphi}(p)  \in \mathcal A \times \mathcal A : d_H^{\oplus} \big(\boldsymbol{\varphi}(p), \boldsymbol{\varphi}^\star(p)\big) \leq \rho(p)\}.$
Assume that \ref{A1}-\ref{A3} and \ref{Assump.W} hold, and $W = \beta \overline{W}$ for some $\beta>0$. 
Suppose there exist $m_i = m_i(\rho,p) \geq 0, i \in \{1, \cdots, 4\}$ such that the following applies for all $\boldsymbol{\varphi}(p) = (\varphi(p), \hat \varphi(p))$, $\boldsymbol{\mu}(p) = (\mu(p), \hat \mu(p)) \in V_{\rho}^p$ with $\mathcal G_p(\boldsymbol{\varphi}(p))=:(\varphi^+(p), \hat \varphi^+(p))$, $\mathcal G_p(\boldsymbol{\mu}(p)) =:(\mu^+(p), \hat \mu^+(p))$.
\begin{enumerate}[label=\textbf{H\arabic*}]
    \item   $\Vert\log \varphi_1^+ \! - \! \log \mu_1^+ \Vert_\infty   \leq m_1 \, d_H^\oplus(\boldsymbol{\varphi},\boldsymbol{\mu})$ and  $\Vert\log \hat \varphi_0^+ \! - \! \log \hat \mu_0^+ \Vert_\infty  \leq m_2 \, d_H^\oplus(\boldsymbol{\varphi},\boldsymbol{\mu})$.\label{eq:assump-thm2-1}
    \item $\sup_t\Vert \nabla \log   \varphi_t \! - \!\nabla \log \mu_t \Vert_\infty \leq m_3 \, d_H^\oplus(\boldsymbol{\varphi},\boldsymbol{\mu})$ and $\sup_t\Vert \nabla \log   \hat \varphi_t \! - \!\nabla \log \hat\mu_t \Vert_\infty \leq m_4 \, d_H^\oplus(\boldsymbol{\varphi},\boldsymbol{\mu})$. \label{eq:assump-thm2-2}
\end{enumerate}
\begin{align}\label{eq:Lambda_p}
\text{Let } \Lambda_p :=  \max \Big\{ 2e( m_1 &+  m_3\sigma^2 \beta\Vert \nabla \overline{W} \Vert_\infty), \nonumber \\
& 2e(m_2 + m_4\sigma^2 \beta\Vert \nabla \overline{W} \Vert_\infty) \Big\}. 
\end{align}
If ${\Lambda_p<1}$, then the map $\mathcal G_p$ is locally contractive near $\boldsymbol{\varphi}^\star(p)$ with respect to $d_H^\oplus(\cdot, \cdot)$.
Additionally, for any $\boldsymbol{\varphi}^{0}(p) \in V_{\rho}^p$, $\boldsymbol{\varphi}^{j+1}(p) = \mathcal G_p(\boldsymbol{\varphi}^j(p)) \to \boldsymbol{\varphi}^\star(p)$ as $j \to \infty$, and $\boldsymbol{\varphi}^\star(p)$ is the unique fixed point in $V_{\rho}^p$. 
\end{theorem}
\begin{proof}
    Let ${\boldsymbol{\varphi}^{j}(p),\boldsymbol{\mu}^{j}(p) \in V_\rho^p}$ with ${\boldsymbol{\varphi}^{j+1}(p) = \mathcal G_p(\boldsymbol{\varphi}^{j}(p))}$, ${\boldsymbol{\mu}^{j+1}(p) = \mathcal G_p(\boldsymbol{\mu}^{j}(p))}$. Note that $\varphi_t^{j+1}, \mu_t^{j+1}$ are limits of the innermost iteration, and therefore, they respectively satisfy \eqref{eq:backward-alg} where $\varphi_t^j, \mu_t^j$ appear in the reaction coefficients with final boundary condition  $\varphi_1^{j+1}, \mu_1^{j+1}$.

    If we set $\nu_t^{j+1}:= \log  \varphi_t^{j+1} -  \log \mu_t^{j+1}$, then from \eqref{eq:HJB-for},
    \begin{align}\label{eq:zeta}
        \partial_t \nu_t^{j+1} &\!-\! \frac{\sigma^2}{2} \big\langle 2(\nabla W*p_t) - \! \nabla \log \varphi_t^{j+1} \nonumber \\
        &-\! \nabla \log \mu_t^{j+1}, \nabla \nu_t^{j+1} \big \rangle + \frac{\sigma^2}{2}  \Delta \nu_t^{j+1} =  F_t^{j},
    \end{align}
where 
${F_t^{j}(x) := - \sigma^2\int_{\Omega}\langle \nabla \nu_t^{j}(y), \nabla W(x-y)\rangle p_t(y) \, \dy}$.
We can apply \cite[Theorem 2.10]{lieberman1996second} again to obtain 
\begin{align} \label{eq:bound-zeta}
 \!\!\!  \!\! \sup_t   \Vert \nu_t^{j+1} \Vert_\infty \leq e  \Big( \Vert \nu_1^{j+1}\Vert_\infty +  \sup_t \Vert F_t^{j} \Vert_\infty \Big),
\end{align}
 provided that the same Dirichlet spatial boundary conditions are enforced on $\varphi^{j+1}, \mu^{j+1}$.
By \ref{eq:assump-thm2-1}, we have $\Vert \nu_1^{j+1} \Vert_\infty \leq m_1 d_H^\oplus({\boldsymbol{\varphi}^j, \boldsymbol{\mu}^j})$.
From \ref{eq:assump-thm2-2} and Young's convolution inequality, we get $\Vert F_t^{j}  \Vert_\infty \leq  \sigma^2 \beta  m_3  \Vert \nabla \overline{W} \Vert_\infty d_H^\oplus({\boldsymbol{\varphi}^j, \boldsymbol{\mu}^j}).$
Therefore, by substituting the bounds on $\Vert \nu_1^{j+1} \Vert_\infty$ and $\Vert F_t^{j}  \Vert_\infty$ into \eqref{eq:bound-zeta}, we have $\sup_t  \Vert \nu_t^{j+1} \Vert_\infty \leq e  (m_1+m_3 \sigma^2 \beta \Vert \nabla \overline{W} \Vert_\infty) d_H^\oplus(\boldsymbol{\varphi}^j,\boldsymbol{\mu}^j)$. 
By \eqref{eq:infty-H}, we arrive at the following bound.
\begin{align}\label{eq:bound-phi-j}
 \sup_t d_H(\varphi_t^{j+1}&,\mu_t^{j+1}) \leq \nonumber\\
  &  2e (m_1 + \sigma^2 \beta m_3 \Vert \nabla \overline{W} \Vert_\infty) d_H^\oplus(\boldsymbol{\varphi}^j,\boldsymbol{\mu}^{j}).
\end{align}

By applying similar arguments, one can get
\begin{align}\label{eq:bound-hatphi-j}
\sup_td_H&(\hat\varphi_t^{j+1},\hat\mu_t^{j+1}) \leq \nonumber\\
& 2e (m_2 + \sigma^2 \beta m_4 \Vert \nabla \overline{W} \Vert_\infty)   d_H^\oplus(\boldsymbol{\varphi}^j,\boldsymbol{\mu}^{j}).
\end{align}
From the definition of  $d_H^{\oplus}$ in \eqref{eq:dH-prod}, the estimates in \eqref{eq:bound-phi-j} and \eqref{eq:bound-hatphi-j} translate into a contraction estimate in $d_H^{\oplus}$, giving
$d_H^\oplus(\mathcal G_p(\boldsymbol{\varphi}),\mathcal G_p(\boldsymbol{\mu})) \leq \Lambda_p d_H^\oplus(\boldsymbol{\varphi},\boldsymbol{\mu}),$
for all $\boldsymbol{\varphi},\boldsymbol{\mu}$ in $V_\rho^p$, where $\Lambda_p$ is as given in \eqref{eq:Lambda_p}.

Let us set ${\boldsymbol{\mu}(p) = \boldsymbol{\varphi}^\star(p)}$ and ${\boldsymbol{\varphi}(p) = \boldsymbol{\varphi}^{0}(p)}$. As long as ${\boldsymbol{\varphi}^{0}(p) \in V_{\rho}^p}$ and $\Lambda_p<1$, we have $d_H^{\oplus} \big(\boldsymbol{\varphi}^1(p),\boldsymbol{\varphi}^\star(p)\big) = d_H^{\oplus} \big(\mathcal G_p(\boldsymbol{\varphi}^0(p)),\mathcal G_p(\boldsymbol{\varphi}^\star(p)) \big)  \leq \Lambda_p \,  d_H^{\oplus} \big(  \boldsymbol{\varphi}^0(p) ,  \boldsymbol{\varphi}^\star(p)\big) < \rho(p)$, and $\boldsymbol{\varphi}^1(p)$ stays in $V_\rho^p$. Hence, $d_H^{\oplus} \big( \boldsymbol{\varphi}^j(p), \boldsymbol{\varphi}^\star(p) \big)  \leq \Lambda_p^j \,  d_H^{\oplus} \big( \boldsymbol{\varphi}^0(p) ,  \boldsymbol{\varphi}^\star(p) \big).$
Further, $\lim_{j \to \infty}d_H^{\oplus}\big(\boldsymbol{\varphi}^j(p), \boldsymbol{\varphi}^\star(p) \big) =0.$
Although $d_H^\oplus$ is projective, the sequence $\boldsymbol{\varphi}^j(p) \to \boldsymbol{\varphi}^\star(p)$ without scaling ambiguity as $\boldsymbol{\varphi}^\star(p)$ has to satisfy the same Dirichlet spatial boundary condition as elements of $V_\rho^p$.
The uniqueness of such a fixed point follows arguments analogous to those at the end of the proof of Theorem~\ref{thm:k-iter}.
\end{proof}

\begin{theorem}\label{thm:overall}
    Let $(p^\star,\varphi^\star(p^\star), \hat \varphi^\star(p^\star))$ be a fixed point consistent with \eqref{eq:map c} and \eqref{eq:map g}.
    Under the assumptions of Theorems~\ref{thm:k-iter} and \ref{thm:j-iter}, if $p^{(0)} \in B_r(p^\star)$ and $\lambda<1$, and if $\boldsymbol{\varphi}^0(p^{(k)}) \in V_\rho^{p^{(k)}}$ for all $k$ \footnote{By \eqref{eq:Bound-varphi} and \eqref{eq:Bound-hatvarphi}, one can verify that this holds if $\lambda(1+ \lambda) r \leq \rho(p^{(k)}),k \geq 1$.} together with $\sup_{p \in B_r(p^\star)}\Lambda_{p}<1$, then, as $k,j \to \infty$, Algorithm~\ref{alg} converges locally in $d_H^\mathcal S \times d_H^{\oplus}$ to $(p^\star,\varphi^\star(p^\star), \hat \varphi^\star(p^\star))$, which is unique in $B_r(p^\star) \times V_\rho ^{p^\star}$. 
\end{theorem}
\begin{proof}
    If $p^{(0)} \in B_r(p^\star)$ and $\lambda<1$, then $p^{(k)} \in B_r(p^\star)$ for all $k$, and $p^{(k+1)} = \mathcal C(p^{(k)}) \to p^\star$ as $k \to \infty$. For a fixed $p=p^{(k)}$, if $\boldsymbol{\varphi}^{0} (p^{(k)})\in V_\rho^{p^{(k)}}$, and $\sup_{p \in B_r(p^\star)}\Lambda_{p}<1$, then $\boldsymbol{\varphi}^{j+1} (p^{(k)}) = \mathcal G_{p^{(k)}} \big(\boldsymbol{\varphi}^{j} (p^{(k)})\big) \to \boldsymbol{\varphi}^\star (p^{(k)})$, as $j \to \infty$. Further, by \eqref{eq:Bound-varphi}, taking $\varphi^\star(p)=\varphi^\star(p^{(k)}), \varphi^\star(q) = \varphi^\star(p^\star)$, $\lim_{k \to \infty} \sup_{t} d_H\big(\varphi_t^\star(p^{(k)}), \varphi_t^\star(p^\star)\big) =0.
    $ Analogously, by \eqref{eq:Bound-hatvarphi}, we get $\lim_{k \to \infty} \sup_{t} d_H\big(\hat\varphi_t^\star(p^{(k)}), \hat\varphi_t^\star(p^\star)\big) =0.$
    Consequently, $\lim_{k \to \infty} \big(p^{(k)},\varphi^\star(p^{(k)}), \hat\varphi^\star(p^{(k)})\big) = \big(p^\star,\varphi^\star(p^\star), \hat\varphi^\star(p^\star)\big)$. Uniqueness of the limit point follows from analogous arguments to those at the end of the proof of Theorem~\ref{thm:k-iter}.
\end{proof}


Theorems~\ref{thm:k-iter}-\ref{thm:overall} can be understood from a perturbative viewpoint, where allowed perturbations within specified neighborhoods are evaluated based on two types of sufficient conditions. The first relates to the regularity of perturbations, such as \ref{Assump.psi_t} and \ref{eq:assump-thm2-2}, where the corresponding constants, e.g., $a_1, m_3$, are weighted by the interaction strength $\beta$ in the expressions of $\lambda, \Lambda_p$,
and hence, their effects are attenuated when $\beta$ is sufficiently small. 
The second pertains to stability, such as \ref{Assump.phipsi-boundary} and \ref{eq:assump-thm2-1}, assessing the sensitivity of the HJB solutions at the boundary to perturbations of its coefficients stemming from updates on $p^{(k)}$, or the reaction coefficients. Weak interaction also moderates these contributions by reducing the drift and reaction terms, cf.~\eqref{eq:HJB-for} and \eqref{eq:HJB-back}. Thus, weak interaction is one way by which $\lambda, \Lambda_p<1$ hold.
Our estimates are conservative, relying on non-tight upper bounds such as Theorem~2.10 in \cite{lieberman1996second} to present simple, verifiable conditions. The local neighborhoods considered are expected to be larger under tighter bounds. In practice, one can also apply damping when updating the reaction coefficients and/or $p_t$, which can help stabilize iterations outside the regimes considered. Analyzing a damped version of the algorithm is a subject of future work.

\section{Numerical Results}\label{sec:NumericalResults}
We present two examples for solving the MFSB problem using Algorithm~\ref{alg}. In both, $p_{\rm in}(x) \propto 0.5\big(\exp \big(-(x-0.5)^2/0.08\big)+\exp \big(-(x+0.4)^2/0.08\big)\big)$, $p_{\rm fin}(x) \propto \exp \big(-(x-0.4)^2/0.08\big)$, {${\texttt{tol}}=10^{-6}$}, and $\sigma^2 =0.5$.  The algorithm is initialized with the non-interacting SB solution. Our code is available in the \href{https://github.com/a-eld/MFSB}{GitHub} repository.

\begin{figure}[t]
    \centering
    \includegraphics[width=1\linewidth]{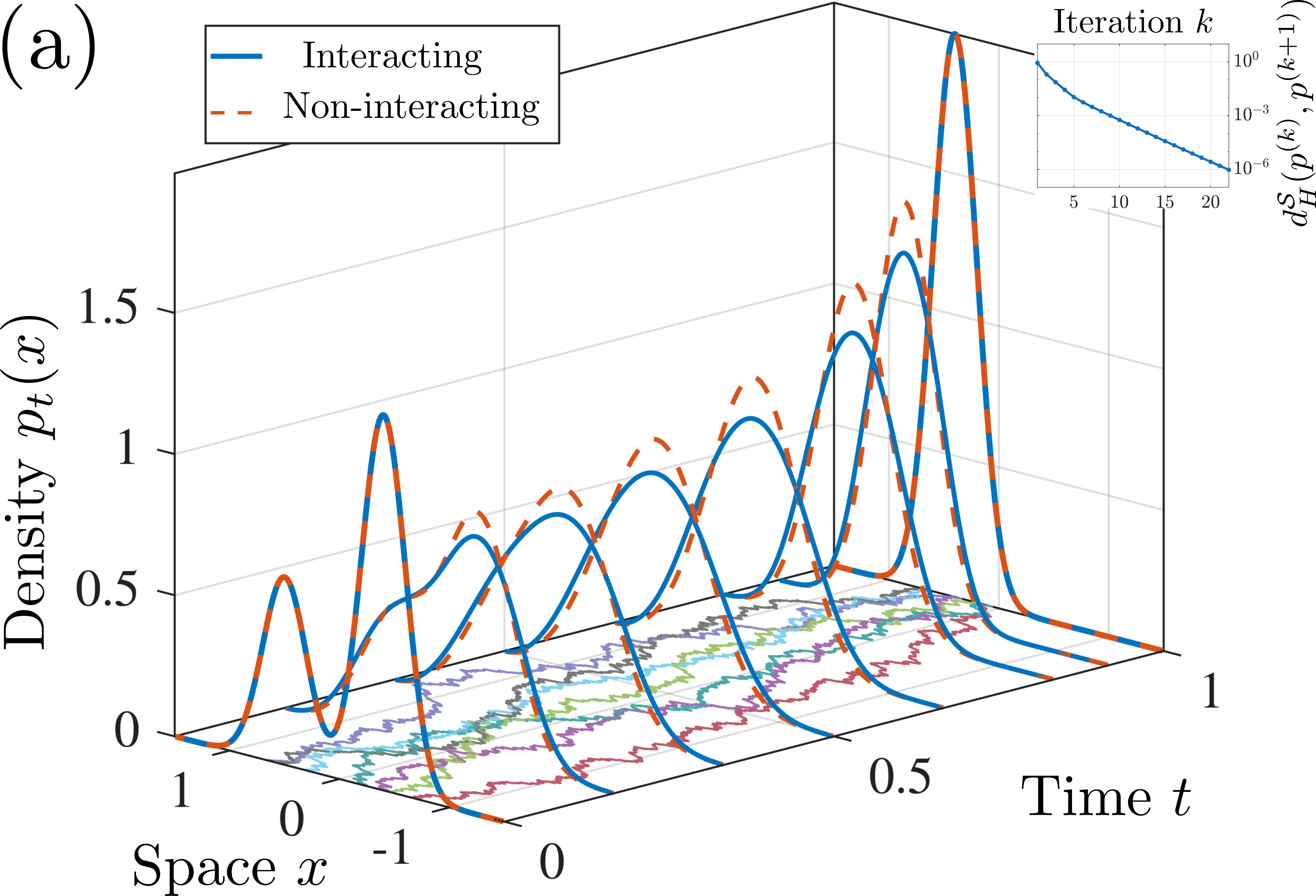}\\[3pt]
    \includegraphics[width=1\linewidth]{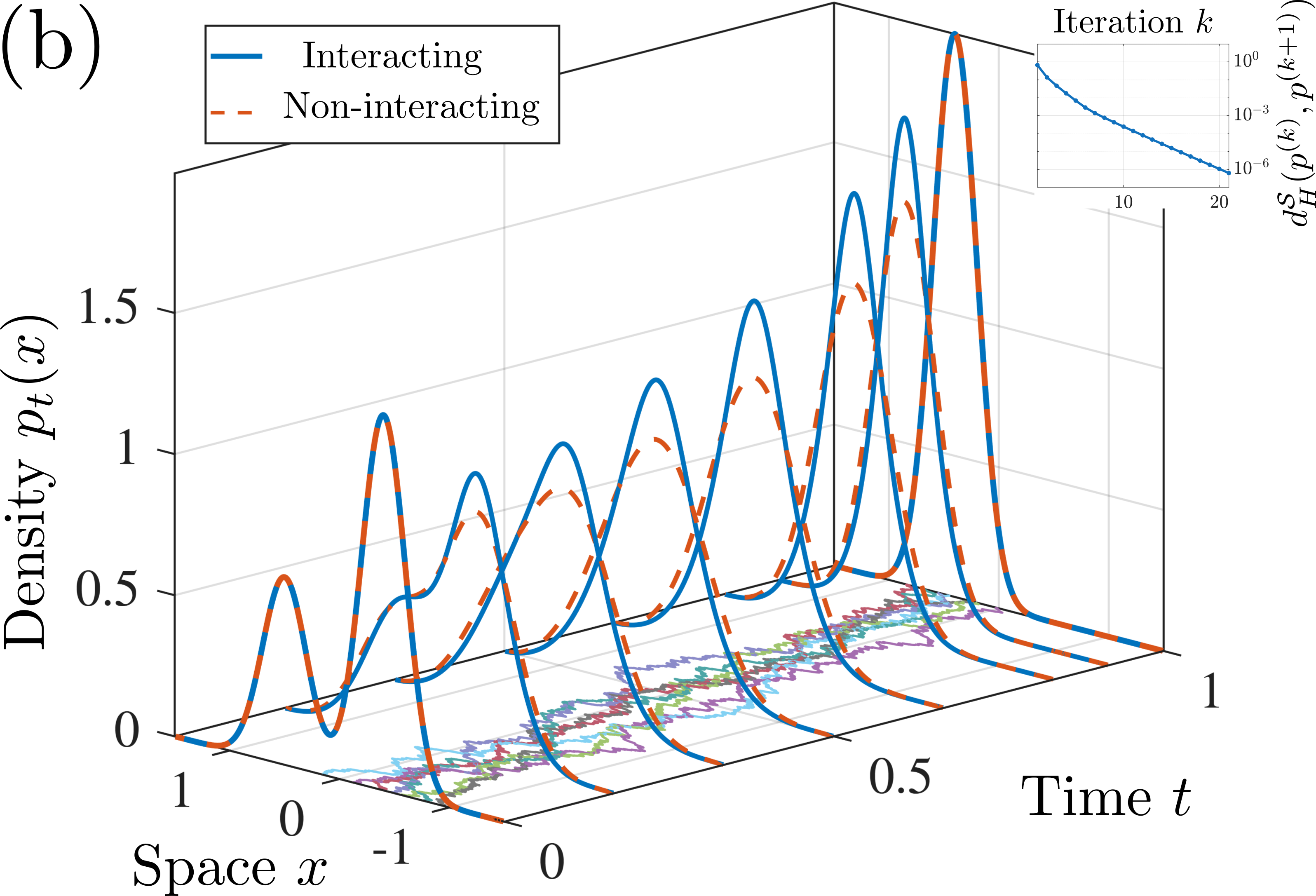}
    \caption{}
    \label{fig}
\end{figure}


We consider repulsive and attractive potentials $W(x) = \pm 1.4/(\sqrt{x^2 +0.01^2})^{0.2}$ and the respective results are shown in Fig.~\ref{fig}(a) and Fig.~\ref{fig}(b). The figures show the MFSB marginals, their non-interacting counterparts, and sample trajectories of \eqref{AgentSDE}, generated by $1000$ particles under the obtained controller. The insets in Fig.~\ref{fig}(a) and Fig.~\ref{fig}(b) depict convergence of the outer loop. At the final outer iteration, the residuals $(\sup_t \vert\alpha_t(p^{(k)}) -1 \vert,d_H(\varphi^{\{i_{\rm f}\}}_1, \varphi^{\{i_{\rm f}+1\}}_1), d_H^\oplus(\boldsymbol{\varphi}^{j_{\rm f}},\boldsymbol{\varphi}^{j_{\rm f}+1})),$ where $i_{\rm f}, j_{\rm f}$ respectively denote the final $i$th and $j$th loop indices, are approximately $(8 \times 10^{-3},1.5 \times 10^{-7},3 \times 10^{-7})$ for the repulsive case and $(3 \times 10^{-3}, 2 \times 10^{-7},3 \times 10^{-7})$ for the attractive case.

\balance
\bibliographystyle{IEEEtran}
\bibliography{references.bib}
\end{document}